\documentclass{amsart}
\usepackage{amsmath,amssymb,amsthm,amscd,enumerate}
\newtheorem{theorem}{Theorem}[section]
\newtheorem{lemma}[theorem]{Lemma}
\newtheorem{proposition}[theorem]{Proposition}
\newtheorem{corollary}[theorem]{Corollary}
\newtheorem{remark}[theorem]{Remark}
\newtheorem{example}[theorem]{Example}

\begin{document}
      
\title[fqf-rings]{Commutative rings whose finitely generated ideals are quasi-flat}
\author{Fran\c{c}ois Couchot}
\address{Normandie Univ, UNICAEN, CNRS, LMNO, 14000 Caen, France}
\email{francois.couchot@unicaen.fr} 


\begin{abstract} A definition of quasi-flat left module is proposed and it is shown that any left module which is either quasi-projective or flat is quasi-flat. A characterization of local commutative rings for which each ideal is quasi-flat (resp. quasi-projective) is given. It is also proven that each commutative ring $R$ whose finitely generated ideals are quasi-flat is of $\lambda$-dimension $\leq 3$, and this dimension $\leq 2$ if $R$ is local. This extends a former result about the class of arithmetical rings. Moreover, if $R$ has a unique minimal prime ideal then its finitely generated ideals are quasi-projective if they are quasi-flat.
\end{abstract}

\maketitle

In \cite{AbJaKa11} Abuhlail, Jarrar and Kabbaj studied the class of commutative fqp-rings ({\bf f}initely generated ideals are {\bf q}uasi-{\bf p}rojective). They proved that this class of rings strictly contains the one of arithmetical rings and is strictly contained in the one of Gaussian rings. It is also shown that the property for a commutative ring to be fqp is preserved by localization. It is known that a commutative ring $R$ is arithmetical (resp. Gaussian) if and only if $R_M$ is arithmetical (resp. Gaussian) for each maximal ideal $M$ of $R$. But an example given in \cite{Cou15} shows that a commutative ring which is a locally fqp-ring is not necessarily a fqp-ring. So, in this cited paper the class of fqf-rings is introduced. Each local commutative fqf-ring is a fqp-ring, and a commutative ring is fqf if and only if it is locally fqf. These fqf-rings are defined in \cite{Cou15} without a definition of quasi-flat modules. Here we propose a definition of these modules and another definition of fqf-ring which is equivalent to the one given in \cite{Cou15}. We also introduce the module property of self-\underline{flatness}. Each quasi-flat module is self-\underline{flat} but we do not know if the converse holds. On the other hand, each flat module is quasi-flat and any finitely generated module is quasi-flat if and only if it is flat modulo its annihilator.

In Section \ref{S:exa} we give a complete characterization of local commutative rings for which each ideal is self-\underline{flat}. These rings $R$ are fqp  and their nilradical $N$ is the subset of zerodivisors of $R$. In the case where $R$ is not a chain ring for which $N=N^2$ and $R_N$ is not coherent every ideal is flat modulo its annihilator. Then in Section \ref{S:qp} we deduce that any ideal of a chain ring (valuation ring) $R$ is quasi-projective if and only if it is almost maximal and each zerodivisor is nilpotent. This complete the results obtained by Hermann in \cite{Her84} on valuation domains.

In Section \ref{S:lambdadim} we show that each commutative fqf-ring is of $\lambda$-dimension $\leq 3$. This extends the result about arithmetical rings obtained in \cite{Cou03}. Moreover it is shown that this $\lambda$-dimension is $\leq 2$ in the local case. But an example of a local Gaussian ring $R$ of $\lambda$-dimension $\geq 3$ is given.

In this paper all rings are associative and commutative (except in the first section) with unity and all modules are unital.

\section{quasi-flat modules: generalities}
Let $R$ be a ring, $M$ a left $R$-module. A left $R$-module $V$ is {\bf $M$-projective} if the natural homomorphism $\mathrm{Hom}_R(V ,M)\rightarrow\mathrm{Hom}_R (V , M /X)$ is surjective for every submodule $X$ of $M$. We say that  $V$ is {\bf quasi-projective} if $V$ is $V$-projective. A ring $R$ is said to be a left {\bf fqp-ring} if every finitely generated left ideal of $R$ is quasi-projective.

We say that $V$ is {\bf $M$-\underline{flat}}\footnote{the module property $M$-flat is generally used to define flat module} if for any epimorphism $p:M\rightarrow M'$, for any homomorphism $u:V\rightarrow M'$ and for any homomorphism $v: G\rightarrow M$, where $M'$ is a left $R$-module and $G$  a finitely presented left $R$-module, there exists a homomorphism $q:G\rightarrow M$ such that $pq=uv$. We call $V$ {\bf quasi-flat} (resp. {\bf self-\underline{flat}}) if $V$ is $V^n$-\underline{flat} for each integer $n>0$ (resp. $n=1$). Clearly each quasi-flat module is self-\underline{flat} but we do not know if the converse holds.

An exact sequence $\mathcal{S}$ of left $R$-modules $0 \rightarrow F \rightarrow E \rightarrow G \rightarrow 0$  is {\bf pure} if it remains exact when tensoring it with any right $R$-module. Then, we say that $F$ is a \textbf{pure} submodule of $E$. Recall that $\mathcal{S}$ is pure if and only if $\mathrm{Hom}_R(M,\mathcal{S})$ is exact for each finitely presented left $R$-module $M$ (\cite[34.5]{Wis91}). When $E$ is flat, then $G$ is flat if and only if $\mathcal{S}$ is pure (\cite[36.5]{Wis91}).

\begin{proposition}\label{P:proj}
Let $R$ be a ring. Then:
\begin{enumerate}
\item each quasi-projective left $R$-module is quasi-flat;
\item each flat left $R$-module is quasi-flat.
\end{enumerate}
\end{proposition}
\begin{proof}
$(1)$. If $V$ is a quasi-projective left $R$-module then, by \cite[18.2(2)]{Wis91}, $V$ is $V^n$-projective for each integer $n>0$.

$(2)$. By \cite[36.8.3]{Wis91} a left $R$-module is flat if and only if it is $M$-\underline{flat} for each left $R$-module $M$.
\end{proof}

\begin{proposition}\label{P:exact}
Let $R$ be a ring, $0\rightarrow A \xrightarrow{t} B\rightarrow C\rightarrow 0$ an exact sequence of left $R$-modules and $V$ a left module. If $V$ is $B$-\underline{flat} then  $V$ is $A$-\underline{flat} and $C$-\underline{flat}.
\end{proposition}
\begin{proof}
Clearly $V$ is $C$-\underline{flat}. Let $p:A\rightarrow A'$ be an epimorphism of left $R$-modules. Consider the following pushout diagram of left $R$-modules:

\[\begin{matrix}
\ \ 0 & {} & \ \ \ 0 & {} & {} \\
\ \ \downarrow & {} & \ \ \ \downarrow & {} & {} \\
\ \ A & \xrightarrow{p} & \ \ \ A' & \rightarrow & 0 \\
{\tiny t}\downarrow & {} & {\tiny t'}\downarrow & {} & {} \\
\ \ B & \xrightarrow{p'} & \ \ \ B' & \rightarrow & 0
\end{matrix}\]
Let $G$ be a finitely presented $R$-module and $V\xrightarrow{u} A'$ and $G\xrightarrow{v} V$ be homomorphisms. Since $V$ is $B$-\underline{flat} there exists a linear map $G\xrightarrow{d} B$ such that $t'uv=p'd$. By \cite[10.7]{Wis91} the above diagram is also a pullback diagram of left $R$-modules, so there exists a homomorphism $G\xrightarrow{q} A$ such that $pq=uv$. Hence $V$ is $A$-\underline{flat}.
\end{proof}

\begin{corollary}\label{C:flatan}
Let $R$ be a ring, $V$ a finitely generated left module and $I$ its annihilator. Then $V$ is flat over $R/I$ if and only if $V$ is quasi-flat.
\end{corollary}
\begin{proof}
If $V$ is flat over $R/I$ then, from Proposition \ref{P:proj} we deduce that it is quasi-flat. Conversely, if $V$ is generated by $n$ elements then $R/I$ is isomorphic to a submodule of $V^n$. It follows that $F=(R/I)^n$ is isomorphic to a submodule of $V^{n^2}$. By Proposition \ref{P:exact} $V$ is $F$-\underline{flat}. Since there exists an epimorphism $p:F\rightarrow V$, we get that $\ker(p)$ is a pure submodule of $F$. Hence $V$ is flat over $R/I$. 
\end{proof}

In section \ref{S:exa} (Corollary \ref{C:exa} and Example \ref{E:exa}) an example of a quasi-flat module (over a commutative ring) which is not flat modulo its annihilator is given.

We say that a ring $R$ is a left {\bf fqf-ring} if each finitely generated left ideal is quasi-flat. By Corollary \ref{C:flatan} this definition is equivalent to the one given in \cite[Section 3]{Cou15}.

\section{quasi-flat ideals over local fqp-rings}\label{S:exa}

In this section $R$ is a commutative ring.

A module $U$ is {\bf uniserial} if its lattice of submodules is totally ordered by inclusion. A ring $R$ is a {\bf chain ring} (or a valuation ring) if it is a uniserial $R$-module.  A chain ring which is an integral domain is a valuation domain. Recall that $R$ is an {\bf IF-ring} if each injective $R$-module is flat.
When $R$ is a chain ring, we denote by $P$ its maximal ideal, by $Z$ its subset of zero-divisors which is a prime ideal, by $N$ its nilradical and by $Q$ its quotient ring $R_Z$.

\begin{lemma}
\label{L:unis} Let $R$ be a chain ring and $U$ an $R$-module. If $U$ is quasi-flat (resp. quasi-projective) then $aU$ is quasi-flat (resp. quasi-projective) too for each $a\in R$.
\end{lemma}
\begin{proof}
We consider the following homomorphisms: $p:(aU)^n\rightarrow U'$, $u:aU\rightarrow U'$ and $v:G\rightarrow aU$ where $U'$ is an $R$-module, $p$ is surjective, $n$ an integer $>0$ and $G$ a finitely presented $R$-module. By \cite[Theorem 1]{War70} $G$ is a direct sum of cyclic submodules. It is easy to see that we may assume that $G$ is cyclic. So $G=R/bR$ for some $b\in R$. If $x=v(1+bR)$ then $bx=0$ and there exists $y\in U$ such that $x=ay$. So, $bay=0$. Let $v':R/baR\rightarrow U$, $u':U\rightarrow U'$ and $p':U^n\rightarrow U'$ be the homomorphisms defined by $v'(r+baR)=ry$ for each $r\in R$, $u'(z)=u(az)$ for each $z\in U$ and $p'(w)=p(aw)$ for each $w\in U^n$. 

The quasi-flatness of $U$ implies that there exists a morphism $q':R/baR\rightarrow U^n$ such that $p'q'=u'v'$. If we put $q(r+bR)=aq'(r+baR)$ for each $r\in R$ then the equalities  $bq(1+bR)=baq'(1+baR)=0$ imply that $q:G\rightarrow (aU)^n$ is a well defined  homomorphism, and we get $pq=uv$.

Now, suppose that $n=1$ and $U$ is quasi-projective. There exists $t':U\rightarrow U$ such that $p't'=u'$. Let $t=t'\vert_{aU}$. Then $pt=u$.
\end{proof}

Let $I$ be a non-zero proper ideal of a chain ring $R$. Then $I^{\sharp}=\{r\in R\mid rI\subset I\}$ is a prime ideal which is called the {\bf top prime ideal} associated to $I$. It is easy to check that $I^{\sharp}=\{r\in R\mid I\subset (I:r)\}$. It follows that $I^{\sharp}/I$ is the inverse image of the set of zerodivisors of $R/I$ by the natural map $R\rightarrow R/I$. So, $Z=0^{\sharp}$.

\begin{proposition}
\label{P:contZ} Let $R$ be a chain ring. Then each proper ideal $I$ satisfying $Z\subset I^{\sharp}$ is flat modulo its annihilator.
\end{proposition}
\begin{proof}
First assume that $Z\subset I$. In this case $I$ is a direct limit of free modules of rank one. So, it is flat. Now suppose that $I\subseteq Z$ and let $t\in I^{\sharp}\setminus Z$ and $a\in I\setminus tI$. Then $a=ts$ for some $s\in Z\setminus I$ and $t\in (I:s)$. So, $Z\subset (I:s)$. It is easy to check that $I=s(I:s)$, $I\cong (I:s)/(0:s)$,  $(0:I)\supseteq (0:s)$ and $(I:s)/(0:s)$ strickly contains $Z/(0:s)$ the subset of zerodivisors of $R/(0:s)$ (see \cite[Lemma 21]{Couch03}).
\end{proof}

\begin{remark}\label{R:P=Z}
\textnormal{If $P=Z$ then by \cite[Lemma 3]{Gil71} and \cite[Proposition 1.3]{KlLe69} we have $(0:(0:I))=I$ for each ideal $I$ which is not of the form $Pt$ for some $t\in R$. In this case $R$ is self FP-injective and the converse holds. So, if $A$ is a proper  ideal such that $A^{\sharp}=P$ then  $R/A$ is self FP-injective and it follows that $(A:(A:I))=I$ for each ideal $I\supseteq A$ which is not of the form $Pt$ for some $t\in R$.}
\end{remark}

\begin{proposition}
\label{P:nZflat} Let $R$ be a chain ring. Then any proper ideal $I$ satisfying $I^{\sharp}\subset Z$ is not self-\underline{flat}.
\end{proposition}
\begin{proof}
Let $s\in Z\setminus I^{\sharp}$. Since $s\notin s^2Q$, by applying the above remark to $Q$ we get that there exists  $a\in (0:s^2)\setminus (0:s)$. The multiplication by $s$ in $I$ induces an isomorphism $\sigma:I/(0:s)\rightarrow I$. Let $u=\sigma^{-1}$, $p:I\rightarrow I/(0:s)$ be the natural epimorphism and $v:R/sR\rightarrow I$ the homomorphism defined by $v(r+sR)=rsa$. Then $uv(1+sR)=a+(0:s)$ and $sb\ne 0$ for each $b\in a+(0:s)$. So, there is no homomorphism $q:R/sR\rightarrow I$ such that $pq=uv$.  
\end{proof}

\begin{lemma}\label{L:FP-inj}
Let $R$ be a chain ring and $I$ a nonzero proper ideal. Assume that $P=Z$ and $I\ne aP$ for each $a\in R$. Then $I$ is FP-injective over $R/A$ where $A=(0:I)$.
\end{lemma}
\begin{proof}
By Remark \ref{R:P=Z} we have $I=(0:A)$. Let $x\in I$ and $c\in R\setminus A$ such that $(A:c)\subseteq (0:x)$. Then $(0:c)\subseteq (0:x)$. Since $R$ is self FP-injective there exists $y\in R$ such that $x=cy$. We have $(0:y)=c(0:x)\supseteq c(A:c)=A$ (the first equality holds by \cite[Lemma 2]{Couch03}). Hence $y\in I$.
\end{proof}

\begin{theorem}
\label{T:fmain} Let $R$ be a chain ring. Assume that either $Z\ne Z^2$ or $Q$ is not coherent. Then the following conditions are equivalent:
\begin{enumerate}
\item $Z=N$;
\item each ideal $I$ is flat over $R/A$ where $A=(0:I)$.
\end{enumerate}
\end{theorem}
\begin{proof}
By Proposition \ref{P:nZflat} $(2)\Rightarrow (1)$.

$(1)\Rightarrow (2)$. Let $I$ be an ideal and $A=(0:I)$. By Proposition \ref{P:contZ} it remains to examine the case where $I^{\sharp}=Z$. If $Z\ne Z^2$ then $Z$ is principal over $Q$. It follows that $Q$ is Artinian. Since $I$ is a principal ideal of $Q$ then $I$ is flat over $Q/A$ and $R/A$. Now suppose that $Z=Z^2$ and $Q$ is not coherent. By \cite[Theorem 10]{Couch03} $Z$ is flat, and we easily deduce that $aZ$ is flat over $R/(0:a)$ for each $a\in R$. Now suppose that $I$ is neither principal over $Q$ nor of the form $aZ$ for each $a\in R$. By Lemma \ref{L:FP-inj} $I$ is FP-injective over $R/A$. From $Q$ no coherent we deduce that $(0:r)$ is not principal over $Q$ for each $0\ne r\in I$. By \cite[Theorem 15(4)(c)]{Cou16} $I$ is flat over $R/A$.
\end{proof}

\begin{remark}\label{R:nQcoh}
\textnormal{If $R$ is a chain ring such that either $Z$ is principal over $Q$ or $Q$ is not coherent then each ideal $I$ satisfying $Z\subseteq I^{\sharp}$ is flat modulo its annihilator.}
\end{remark}

\begin{lemma}\label{L:inter}
Let $R$ be a chain ring and $M$ a finitely generated $R$-module. Then, for each proper ideal $A$ which is not of the form $rP$ for any $r\in P$, we have $AM=\cap_{s\in P\setminus A}sM$.
\end{lemma}
\begin{proof}
By \cite[Theorem 15]{FuSa85} there is a finite sequence of pure submodules of $M$, \[0=M_0\subset M_1\subset\dots\subset M_{n-1}\subset M_n=M.\] such that $M_k/M_{k-1}$ is cyclic for each $k=1,\dots,n$. We proceed by induction on $n$. When $n=1$  $M$ is cyclic and we use \cite[Lemma 29]{Couch03} to conclude. Now suppose that $n>1$. Let $x\in\cap_{s\in P\setminus A}sM$. We may assume that $x\notin M_{n-1}$. Since $M/M_{n-1}$ is cyclic there exist $y\in M$ and $a\in A$ such that $(x-ay)\in M_{n-1}$. Moreover, by using the fact that $M_{n-1}$ is a pure submodule of $M$ we have that $(x-ay)\in\cap_{s\in P\setminus A}sM_{n-1}$. From the induction hypothesis we deduce that $x=ay+bz$ for some $z\in M_{n-1}$ and $b\in A$.  
\end{proof}

\begin{proposition}
\label{P:idemp} Let $R$ be a chain ring. Then, for each $a\in R$, $aZ$ is quasi-flat.
\end{proposition}
\begin{proof} We may assume that $Z=Z^2\ne 0$. By Lemma \ref{L:unis} it is enough to study the case  $a=1$. First suppose that $Z=P$. We consider the following homomorphisms: $p:Z^n\rightarrow Z'$, $u:Z\rightarrow Z'$ and $v:G\rightarrow Z$ where $Z'$ is an $R$-module, $p$ is surjective, $n$ an integer $>0$ and $G=R/aR$ for some $a\in R$. If $r=v(1+aR)$ then $ar=0$. Let $s\in Z\setminus Rr$. Then $r=ss'$ for some $s'\in P$ and $u(r)=su(s')$. So, $u(r)\in\cap_{s\in Z\setminus Rr}sZ'$. Consider the following commutative pushout diagram:

\[\begin{matrix}
\ \ 0 & {} & \ \ \ 0 & {} & {} \\
\ \ \downarrow & {} & \ \ \ \downarrow & {} & {} \\
\ \ Z^n & \xrightarrow{p} & \ \ \ Z' & \rightarrow & 0 \\
{\tiny t}\downarrow & {} & {\tiny t'}\downarrow & {} & {} \\
\ \ R^n & \xrightarrow{p'} & \ \ \ R' & \rightarrow & 0
\end{matrix}\]
\\
where $t$ is the canonical inclusion. Clearly $R'$ is finitely generated. So, by Lemma \ref{L:inter} $u(r)=rx'$ for some $x'\in R'$. Let $x\in R^n$ such that $p'(x)=x'$. Let $q:G\rightarrow Z^n$ be the homomorphism defined by $q(1+aR)=rx$. Then $q:G\rightarrow Z^n$  is well defined because $ar=0$ and we have $pq=uv$. Hence $P$ is quasi-flat.

Now assume that $Z\ne P$ and $Z$ is faithful. Let $a\in P$ and $t\in Z\setminus (0:a)$. Let  $K=\ker(p)$ and $G_1=R/Rta$. Then $G\cong Rt/Rta\subseteq G_1$. Since $Q$ is FP-injective $v$ extends to $v_1:G_1\rightarrow Q$. But $v_1(1+Rta)\in Z$ because it is annihilated by $ta$. There exists a homomorphism $q':(G_1)_Z\rightarrow Z^n$ such that $p_Zq'=u_Z(v_1)_Z$. Let $q_1$ be the composition of the natural map $G_1\rightarrow (G_1)_Z$ with $q'$, $r=v_1(1+atR)$ and $x=q_1(1+atR)$. Then $u(r)-p(x)\in K_Z/K$. Let $q=q_1\vert_G$. We have $v(t+taR)=tr$ and $q(t+taR)=tx$. So, $u(tr)-p(tx)=t(u(r)-p(x))=0$. Hence $uv=pq$.
\end{proof}

\begin{proposition}
Let $R$ be a chain ring. Then each ideal $I$ satisfying $Z\subseteq I^{\sharp}$ is self-\underline{flat}.
\end{proposition}
\begin{proof}
By Remark \ref{R:nQcoh} we may assume that $0\ne Z=Z^2$ and $Q$ is coherent, and by Proposition \ref{P:contZ} that $I^{\sharp}=Z$. We may suppose that $I$ is neither principal over $Q$ nor of the form $aZ$ for any $a\in R$. We consider the following homomorphisms: $p:I\rightarrow I'$, $u:I\rightarrow I'$ and $v:G\rightarrow I$ where $I'$ is an $R$-module, $p$ is surjective and $G=R/aR$ for some $a\in P$. By Lemma \ref{L:FP-inj} $I$ is FP-injective over $R/A$, where $A=(0:I)$ and by \cite[Theorem 15(4)(c)]{Cou16} $Z\otimes_RI$ is flat over $R/A$ because $(0:r)$ is principal over $Q$ for each $r\in I$. Since $I=ZI$ the canonical homomorphism $\phi:Z\otimes_RI\rightarrow I$ is surjective. Let $r=v(1+aR)$. Then $r=\phi(s\otimes b)$ where $s\in Z$ and $b\in I$. Since $ar=0$ then $a(s\otimes b)\in\ker(\phi)\cong\mathrm{Tor}^Q_1(Q/Z,I)$. So, $aZ\subseteq (0:s\otimes b)$. Let $v':R/taR\rightarrow Z\otimes_RI$ be the homomorphism by $v'(1+taR)=s\otimes b$ where $t\in Z$. From the flatness of $Z\otimes_R I$ we deduce there exists $q_t':R/taR\rightarrow I$ such that $pq_t'=u\phi v'$. Let $x_t=q_t'(1+taR)$. Then $tax_t=0$. Let $t'$ another element of $Z$. Thus $p(x_t)=p(x_{t'})=u(r)$, whence $(x_{t'}-x_t)\in\ker(p)\subset Qx_t$ since $I$ is a uniserial $Q$-module. So, $Qx_t=Qx_{t'}$ and we can choose $x_t=x_{t'}=x$. Hence $aZ\subseteq (0:x)$. But $(0:x)$ is a principal ideal of $Q$, whence $ax=0$. If we put $q(c+aR)=cx$ for each $c\in R$ then $pq=uv$.
\end{proof}

\begin{theorem}\label{T:fond}
Let $R$ be a chain ring. Then each ideal $I$ is self-\underline{flat} if and only if $Z=N$.
\end{theorem}

The following can be proven by  using \cite[Lemmas 3.8, 3.12 and 4.5]{AbJaKa11}.
\begin{theorem}
\label{T:localfqp} \cite[Theorem 4.1]{Cou15}. Let $R$ a local ring and $N$ its nilradical. Then $R$ is a fqp-ring if and only if either $R$ is a chain ring or $R/N$ is a valuation domain and $N$ is a divisible torsionfree $R/N$-module.
\end{theorem}

\begin{corollary}
\label{C:qmain} Let $R$ be a local fqf-ring which is not a chain ring and $N$ its nilradical. Then each ideal of $R$ is flat modulo its annihilator.
\end{corollary}
\begin{proof}
Let $I$ be an ideal. If $I\subseteq N$ then $I$ is a torsionfree module over the valuation domain $R/N$. Hence it is a flat $R/N$-module. If $I\nsubseteq N$ then each finitely subideal of $I$ is principal and free. So, $I$ is flat.
\end{proof}

The following corollary and example allow us to see there exist quasi-flat modules which are not flat modulo their annihilator.

\begin{corollary}\label{C:exa}
Let $R$ a chain ring. Assume that $P$ is not principal and $R$ is an IF-ring. Then, for each $a\in R$, $aP$ is quasi-flat but it is not flat over $R/(0:aP)$.
\end{corollary}
\begin{proof}
Since $R$ is coherent and $P$ is not finitely generated we get that $P$ is faithful. By \cite[Theorem 10]{Couch03} $P$ is not flat. Let $0\ne a\in P$. There exists $b\in P$ such that $(0:a)=Rb$. So, $aP\cong P/Rb$, and $Rb=(0:aP)$ because $P$ is faithful. By \cite[Theorem 11]{Couch03} $R/Rb$ is an IF-ring and consequently $R/Rb$ satisfies the same conditions as $R$. Hence $aP$ is quasi-flat but not flat over $R/Rb$. 
\end{proof}

\begin{example}\label{E:exa}
Let $R=D/dD$, where $D$ is a valuation domain with a non-principal maximal ideal and $d$ a non-zero element of $D$ which is not invertible. Then $R$ satisfies the assumptions of Corollary \ref{C:exa}.
\end{example}

\section{Quasi-projective ideals over local fqp-rings}\label{S:qp}

An $R$-module $M$ is said to be {\bf linearly compact} if every finitely solvable set of congruences $x \equiv x_{\alpha}$ (mod $M_{\alpha}$) ($\alpha \in \Lambda$,  $x_{\alpha}\in M$ and  $M_{\alpha}$ is a submodule of $M$ for each $\alpha \in \Lambda$) has a simultaneous solution in $M$. A  chain ring $R$ is \textbf{maximal} if it is linearly compact over itself and $R$ is \textbf{almost maximal} if $R/A$ is maximal for each non-zero ideal $A$.

\begin{theorem}\label{T:max}
Let $R$ be a chain ring. The following conditions are equivalent:
\begin{enumerate}
\item $R$ is almost maximal and  $Z=N$;
\item each ideal is quasi-projective.
\end{enumerate}
\end{theorem}
\begin{proof}
$(1)\Rightarrow (2)$.
Let $I$ be a non-zero proper ideal of $R$, $p:I\rightarrow I'$ an epimorphism, $K=\ker(p)$ and $u:I\rightarrow I'$ a homomorphism. First suppose that $I\subseteq Z$. By Theorem \ref{T:fond} $I$ is self-\underline{flat}. So, for each $r\in I$ and $b\in (0:r)$ there exists $y_{r,b}\in I$ such that $by_{r,b}=0$ and $u(r)=p(y_{r,b})$. Even if $K\ne 0$ we can take $y_{r,b}=y_{r,c}=y_r$ if $c$ is another element of $(0:r)$. So, $(0:r)\subseteq (0:y_r)$. Since $Q$ is FP-injective then $y_r=rx_r$ where $x_r\in Q$. We put $R'=Q/K$, $p':Q\rightarrow R'$ the canonical epimorphism and $x_r'=p'(x_r)$ for each $r\in I$. So, for each $r\in I$, $u(r)=rx'_r$. If $s\in I\setminus Rr$ then we easily check that $(x'_s-x'_r)\in R'[r]=\{y\in R'\mid ry=0\}$. If $R$ is almost maximal then the family of cosets $(x'_r+R'[r])_{r\in I}$ has a non-empty intersection. Let $x'$ be an element of this intersection. Then $u(r)=rx'$ for each $r\in I$.  Let $x\in Q$ such that $p'(x)=x'$. For each $r\in I$, $rx\in rx_r+K\subseteq I$. If $q$ is the multiplication by $x$ in $I$ then $pq=u$. Hence $I$ is quasi-projective. Now suppose that $Z\subset I$. Then for each $r\in I\setminus Z$ there exists $y_r\in I$ such that $u(r)=p(y_r)$. But $y_r=r(r^{-1}y_r)=rx_r$ where $x_r\in Q$. We do as above to show that $I$ is quasi-projective.
\end{proof}

\begin{proposition}\label{P:ZP}
Let $R$ a chain ring. Assume that $P=N$. Then $R$ is almost maximal if each ideal $I$ is quasi-projective.
\end{proposition}
\begin{proof}
If $P$ is finitely generated then $R$ is Artinian. In this case $R$ is maximal. Now assume that $P$ is not finitely generated. Let $(a_{\lambda}+I_{\lambda})_{\lambda\in\Lambda}$ be a totally ordered family of cosets such that $I=\cap_{\lambda\in\Lambda}I_{\lambda}\ne 0$. By \cite[Lemma 29]{Couch03} $I\ne aP$ for each $a\in R$. Let $A=P(0:I)$. 

First we assume that $I$ is different of the minimal non-zero ideal when it exists. So, $A\subset P$. We have $I=(0:A)=\cap_{r\in A}(0:r)$ (if $(0:I)$ is not principal then $A=(0:I)$. If not, either $I$ is not principal and from $I=PI$ we deduce that $(0:(0:I))=(0:A)$, or $I$ is principal which implies that $P$ is faithful and  $(0:(0:I))=(0:A)$). Let $r\in A$. We may assume that $I\subset I_{\lambda}$ for each $\lambda\in\Lambda$. Hence there exists  $\lambda\in\Lambda$ such that $I_{\lambda}\subseteq (0:r)$. We put $a(r)=a_{\lambda}r$. If $I_{\mu}\subset I_{\lambda}$ then $(a_{\mu}-a_{\lambda})\in I_{\lambda}$, whence $a_{\mu}r=a_{\lambda}r$. So, in this manner, we define an endomorphism of $A$. Since $P=N$ there exists $c\in P\setminus A$ such that $c^2\in A$. Let $B=(A:c)$. Then $A=cB$ and $c\in B$. Let $p:B\rightarrow A$ be the homomorphism defined by $p(r)=cr$ and $u:B\rightarrow A$ be the homomorphism defined by $u(r)=a(cr)$, for each $r\in B$. The quasi-projectivity of $B$ implies there exists an endomorphism $q$ of $B$ such that $pq=u$. Since $(0:c)\subseteq (0:q(c))$ and $R$ is self FP-injective we deduce that $q(c)=ca'$ for some $a'\in R$ and $a(cr)=cq(r)=q(cr)=rq(c)=a'cr$ for each $r\in B$. Let $\lambda\in\Lambda$. We have $I=\cap_{r\in B}(0:rc)$. Since $I\subset I_{\lambda}$ then $(0:rc)\subseteq I_{\lambda}$ for some $r\in B$. From $I=\cap_{\mu\in\Lambda}I_{\mu}$ we deduce there exists $\mu\in\Lambda$ such that $I_{\mu}\subseteq (0:rc)$. It follows that $(a'-a_{\mu})\in (0:rc)$. But $(a_{\mu}-a_{\lambda})\in I_{\lambda}$, so $a'\in (a_{\lambda}+I_{\lambda})$ for each $\lambda\in\Lambda$. 

Now we assume that $I$ is the minimal non-zero ideal of $R$. In this case $A=P$. Let $s,\ t\in P$ such that $I=Rst$. There exists $\lambda_0\in\Lambda$ such that $I_{\lambda_0}\subseteq Rt\subset P$. Let $\Lambda'=\{\lambda\in\Lambda\mid I_{\lambda}\subseteq I_{\lambda_0}\}$ . Put $J_{\lambda}=(I_{\lambda}: t)$ and $J=\cap_{\lambda\in\Lambda'}J_{\lambda}$. Since $s\in J\setminus I$ then $J$ is not minimal. From $(a_{\lambda}-a_{\lambda_0})\in I_{\lambda_0}$  we deduce there exists $b_{\lambda}\in R$ such that $(a_{\lambda}-a_{\lambda_0})=tb_{\lambda}$ for all $\lambda\in\Lambda'$. If
 $\lambda,\ \mu\in\Lambda'$ such that $I_{\mu}\subseteq I_{\lambda}$ then we easily check that $b_{\mu}\in b_{\lambda}+J_{\lambda}$. From above it follows that there exists $b\in\cap_{\lambda\in\Lambda'}b_{\lambda}+J_{\lambda}$ and it is easy to see that $(a_{\lambda_0}+tb)\in\cap_{\lambda\in\Lambda}a_{\lambda}+I_{\lambda}$. Hence $R$ is almost maximal.
\end{proof}

\begin{proof}[Proof of $(2)\Rightarrow (1)$ in Theorem \ref{T:max}.] Since each ideal $I$ is self-\underline{flat} we have $Z=N$ by Theorem \ref{T:fond}. From the previous proposition we deduce that $Q$ is almost maximal and we may assume that $P\ne Z$. When $Z=0$ $R$ is almost maximal by \cite[Theorem 3.3]{Her84}. Now suppose $Z\ne 0$. We shall prove that $R/Z$ is maximal and we will conclude that $R$ is almost maximal by using \cite[Theorem 22]{Couc06}. Let $0\ne x\in Z$ and $I$ a proper ideal of $R$ such that $Z\subset I$. Since $I$ is quasi-projective then $I$ is $(Qx/Zx)$-projective by \cite[18.2]{Wis91}. Let $q:I\rightarrow Q/Z$ be a homomorphism. If $z\in Z$ and $t\in I\setminus Z$ then $z=z't$ for some $z'\in Z$. So, $q(z)=z'q(t)=0$, whence $q$ factors through $I/Z$. It follows that $I/Z$ is $(Q/Z)$-projective for each ideal $I$ containing $Z$. By \cite[Theorem 3.3]{Her84} $R/Z$ is almost maximal. Suppose that $Z^2=Z$. We have that $Qx$ is $(Qx/Zx)$-projective. Let $q:Q\rightarrow Q/Z$ be a homomorphism and $z\in Z$. There exist $z',\ t\in Z$ such that $z=z't$. So, $q(z)=z'q(t)=0$ whence $q$ factors through $Q/Z$. It follows that $Q/Z$ is quasi-projective. If $Z\ne Z^2$ then $Z$ is principal over $Q$ and there exists $x\in Z$ such that $Z=(0:x)$. Hence $Q/Z\cong Qx$ is quasi-projective. From $R/Z$ almost maximal and \cite[Theorem 3.4]{Her84} we get that $R/Z$ is maximal.
\end{proof}

\begin{theorem}
Let $R$ be a local fqp-ring which is not a chain ring and $N$ its nilradical. Consider the following conditions:
\begin{enumerate}
\item $R$ is a linearly compact ring;
\item each ideal is quasi-projective;
\item $N$ is of finite rank over $R/N$.
\end{enumerate}
Then $(1)\Leftrightarrow ((2)\ \mathrm{and}\ (3))$.
\end{theorem}
\begin{proof}
$(1)\Rightarrow$ ($(2)$ and $(3)$). Let $R'=R/N$ and $Q'$ its quotient field. Then $R'$ is a maximal valuation domain. Since $N$ is a direct sum of modules isomorphic to $Q'$ and a linearly compact module then $N$ is of finite rank by \cite[29.8]{Wis91}. Let $I$ be an ideal contained in $N$. By \cite[Lemma 4.4]{Her84} $I$ is quasi-projective. Now suppose that $I\nsubseteq N$. In this case $N\subset I$ and since $I/N$ is uniserial, by a similar proof as the one of Theorem \ref{T:max} we show that $I$ is quasi-projective.

($(2)$ and $(3)$)$\Rightarrow (1)$. Since $Q'$ is isomorphic to a direct summand of $N$,  $Q'$ and its submodules are quasi-projective. Hence, by \cite[Theorems 3.3 and 3.4]{Her84} $R'$ is maximal. It follows that $N$ is linearly compact. Whence $R$ is linearly compact by \cite[29.8]{Wis91}.
\end{proof}

\begin{remark}
\textnormal{In the previous theorem:
\begin{enumerate}
\item if $N$ is the maximal ideal of $R$ then each ideal is quasi-projective even if $N$ is not of finite rank over $R/N$; 
\item if $N$ is not the maximal ideal and if $Q/N$ is countably generated over $R/N$ then $(1)\Leftrightarrow (2)$ because $(2)\Rightarrow (3)$ by \cite[Lemma 4.3(b)]{Her84}.
\end{enumerate}}
\end{remark}

\section{$\lambda$-dimension of commutative fqf-rings}\label{S:lambdadim}

In this section $R$ is a commutative ring. We say that $R$ is {\bf arithmetical} if $R_P$ is a  chain ring for each maximal ideal $P$.

An $R$-module $E$ is said to be of {\bf finite $n$-presentation} if there exists an exact sequence:
\[F_n \rightarrow F_{n-1} \rightarrow\cdots F_1 \rightarrow F_0 \rightarrow E \rightarrow 0\]
with the $F_i$'s free $R$-modules of finite rank. We write 
\[\lambda_R(E) = \sup \{ n \mid \mathrm{there\ is\ a\ finite}\ n\mathrm{-presentation\ of}\ E\}.\] If $E$ is not finitely generated we also put $\lambda_R(E)= - 1$.

The {\bf $\lambda$-dimension} of a ring $R$ ($\lambda\mathrm{-dim}(R)$ is the least integer $n$ (or $\infty$ if none such exists) such that $\lambda_R(E) \geq n$ implies $\lambda_R(E) =\infty$. See \cite[chapter 8]{Vas76}. Recall that $R$ is Noetherian if and only if $\lambda$-$dim (R) = 0$ and $R$ is coherent if and only if $\lambda$-$dim (R) \leq 1$. The rings of $\lambda$-dimension $\leq n$ are also called {\bf $n$-coherent} by some authors.

This notion of $\lambda$-dimension of a
ring was formulated in \cite[chapter 8]{Vas76} to study the rings of polynomials or power series over a coherent ring.

\begin{theorem}\label{T:localdim}
Let $R$ be a local fqp-ring. Then $\lambda\mathrm{-dim}(R)\leq 2$.
\end{theorem}
\begin{proof}
By \cite[Theorem II.11]{Cou03} $\lambda\mathrm{-dim}(R)\leq 2$ if $R$ is a chain ring. Theorem \ref{T:localfqp} and the following proposition complete the proof.
\end{proof}

\begin{proposition}\label{P:localdim}
Let $R$ be a local fqp-ring which is not a chain ring and $N$ its nilradical. Then:
\begin{enumerate}
\item either $R$ is Artinian or $\lambda\mathrm{-dim}(R)=2$ if $N$ is maximal;
\item $\lambda\mathrm{-dim}(R)=2$ if $N$ is not maximal.
\end{enumerate}
Moreover, if $G$ is a finitely $2$-presented module then:
\begin{enumerate}
\item[(3)] $G$ is free if $N$ is maximal and not finitely generated;
\item[(4)] $G$ is of projective dimension $\leq 1$ if $N$ is not maximal.
\end{enumerate}
\end{proposition}
\begin{proof}
$(1)$ and $(3)$. If $N$ is an $R/N$-vector space of finite dimension then $R$ is Artinian. Assume that $N$ is not of finite dimension over $R/N$. Let $G$ be an $R$-module of finite $2$-presentation. So, there exists an exact sequence
\[F_2\xrightarrow{u_2} F_1\xrightarrow{u_1} F_0\rightarrow G\rightarrow 0,\]
where $F_i$ is free of finite rank for $i=0,1,2$. Let $G_i$ be the image of $u_i$ for $i=1,2$. Since $R$ is local we may assume that $G_i\subseteq NF_{i-1}$ for $i=1,2$. Then $G_i$ is a module of finite length for $i=1,2$. It follows that $F_1$ is of finite length too. This is possible only if $F_1=0$. So, $G$ is free. Hence $\lambda_R(G)=\infty$ and $\lambda\mathrm{-dim}(R)\leq 2$. Let $0\ne r\in N$. Since $(0:r)=N$, $\lambda_R(R/rR)=1$, whence $\lambda\mathrm{-dim}(R)=2$.

$(2)$ and $(4)$. Let $P$ be the maximal ideal of $R$. Each $r\in P\setminus N$ is regular. So, $Rr$ is free and since each finitely generated ideal which is not contained in $N$ is principal, $P$ is flat. Let $G$, $G_1$ and $G_2$ be as in $(1)$. Since $R$ is local we may assume that $G_i\subseteq PF_{i-1}$ for $i=1,2$. Then $\mathrm{Tor}^R_1(G_1,R/P)\cong\mathrm{Tor}^R_2(G,R/P)=0$. So, the following sequence is exact:
\[0\rightarrow G_2/PG_2\rightarrow F_1/PF_1\xrightarrow{v} G_1/PG_1\rightarrow 0,\]
where $v$ is induced by $u_1$. Since $v$ is an isomorphism it follows that $G_2/PG_2=0$, and by Nakayama Lemma $G_2=0$. So, $G_1$ is free. Now, we do as in $(1)$ to conclude ($N$ is not finitely generated because it is divisible over $R/N$).
\end{proof}

Let $A$ be a ring and $E$ an $A$-module. The {\bf trivial ring extension} of $A$ by $E$ (also called the idealization of $E$ over $A$) is the ring $R := A\propto E$ whose underlying group is $A\times E$ with multiplication given by $(a, e)(a', e') = (aa', ae' + a'e)$. Let $R$ be a ring. For a polynomial $f\in R[X]$, denote  by $c(f)$ (the content of $f$) the ideal of $R$ generated by the coefficients of $f$. We say that $R$ is {\bf Gaussian} if $c(fg)=c(f)c(g)$ for any two polynomials $f$ and $g$ in $R[X]$. 

The following example shows that we cannot replace "fqf-ring" with "Gaussian ring" in Theorem \ref{T:localdim}.

\begin{example}
Let $D$ be a non almost maximal valuation domain and $M$ its maximal ideal. Let $0\ne d\in M$ such that $D/Dd$ is not maximal and $E$ the injective hull of $D/Dd$. Consider $R=D\propto E$ the trivial ring extension of $D$ by $E$. Then $R$ is a Gaussian local ring and $\lambda\mathrm{-dim}(R)\geq 3$.
\end{example}
\begin{proof}
By \cite[Theorem 17]{Couch03} $E$ is not uniserial. By \cite[Corollary 4.3]{Cou15} $R$ is Gaussian but not a fqp-ring because $E$ is neither uniserial nor torsionfree. Let $e\in E$ such that $(0:e)=Dd$. We put $a=(0,e)$ and $b=(d,0)$. Then $(0:a)=Rb$ and  $(0:b)=\{(0,x)\mid dx=0\}=0\propto E[d]$, where $E[d]=\{x\in E\mid dx=0\}$. If $D'=D/Dd$ then $E[d]$ is isomorphic to the injective hull of $D'$ over $D'$ and $D'\ne E[d]$ because $D'$ is not maximal. By \cite[Theorem 11]{Couch03} $D'$ is an IF-ring and consequently $E[d]$ and $E[d]/D'$ are flat over $D'$. Then $E[d]$ is not finitely generated, else $E[d]/D'$ is a free $D'$-module and this contradicts that $E[d]$ is an essential extension of $D'$. So, $(0:b)$ is not finitely generated, $\lambda_R(R/Ra)=2$ and $\lambda\mathrm{-dim}(R)\geq 3$.
\end{proof}

\begin{proposition}
Let $R$ be a fqf-ring with a unique minimal prime ideal $N$. The following assertions hold:
\begin{enumerate}
\item $R_P$ is not a chain ring for each maximal ideal $P$ if $R$ is not arithmetical;
\item $R$ is a fqp-ring.
\end{enumerate}
\end{proposition}
\begin{proof}
$(1)$. There exists a maximal ideal $L$ such that $R_L$ is a local fqp-ring which is not a chain ring. So, $N_L$ is torsionfree and divisible over $R_L/N_L$. Moreover, since $N_L$ is not uniserial over $R_L$, by \cite[Lemma 3.8]{AbJaKa11} there exist $a,b\in N_L$ such that $aR_L\cap bR_L=0$. It follows that $N_L=N_N$ and it is a vector space over $R_N/N_N$ of dimension $>1$. Let $P$ be a maximal ideal. Then $N_N$ is a localization of $N_P$. Consequently $N_P$ is not uniserial. Hence, $R_P$ is not a chain ring.

$(2)$. It follows that $N$ is a torsionfree divisible module over $R/N$. So, if $I$ is a finitely generated ideal contained in $N$ then $I$ is a finitely generated flat module over the Pr\"ufer domain $R/N$. So, $I$ is projective over $R/N$. Now, if $I\nsubseteq N$ then $I_P$ is a free $R_P$-module of rank $1$. We conclude by \cite[Chap.2, \S 5, 3, Th\'eor\`eme 2]{Bou61} that $I$ is projective.
\end{proof}

\begin{corollary}\label{C:fqfdim}
Let $R$ be a fqf-ring with a unique minimal prime ideal $N$. Assume that $R$ is not arithmetical. Then either $R$ is Artinian or $\lambda\mathrm{-dim}(R)=2$.
\end{corollary}
\begin{proof} When $N$ is maximal we use Proposition \ref{P:localdim}. Now assume that $N$ is not maximal. Let $G$ be a $R$-module such that $\lambda_R(E)\geq 2$. We use the same notations as in the proof of Proposition \ref{P:localdim}. This proposition implies that $G_1$ is locally free. Since $G_1$ is a finitely presented flat module, we successively deduce that $G_1$ is projective, $G_2$ is projective and $\ker(u_2)$ is finitely generated. 
\end{proof}

An integral domain $D$ is {\bf almost Dedekind} is $D_P$ is a Noetherian valuation domain for each maximal ideal $P$. 

The following example shows that we cannot remove the assumption "$R$ is not arithmetical" in Corollary \ref{C:fqfdim}.
\begin{example}\label{E:ex3}
Let $D$ be an almost Dedekind domain which is not Noetherian (see \cite[Example III.5.5]{FuSa01}), $Q$ its quotient field, $P'$ a maximal ideal of $D$ which is not finitely generated and $E=Q/D_{P'}$. Let $R=D\propto E$ and $N=\{(0, y)\mid y\in E\}$. Then $R$ is an arithmetical ring, $N$ is its unique minimal prime ideal and $\lambda\mathrm{-dim}\ R=3$. Moreover, $R_P$ is IF where $P$ is the maximal ideal of $R$ satisfying $P'=P/N$, and $R_L$ is a valuation domain for each maximal ideal $L\ne P$.
\end{example}
\begin{proof}
For each maximal ideal $L$ of $R$ let $L'=L/N$. Let $p\in P'$ such that $P'D_{P'}=pD_{P'}$, $x=1/p+D_{P'}$,  $a=(p,0)$ and $b=(0, x)$. Since $0$ is the sole prime ideal of $D$ contained in $P'\cap L'$ for each maximal ideal $L'\ne P'$, then $E_{L'}=0$. So, $R_L=D_{L'}$. Since $E$ is  uniserial and divisible over $D_{P'}$, $R_P$ is a chain ring by \cite[Proposition 1.1]{Cou15}. So, $R$ is arithmetical. We have $(0:_{R_P}b)=aR_P=PR_P$. By \cite[Theorem 10]{Couch03} $R_P$ is IF. Clearly $Dx$ is the minimal submodule of $E$. So, $P'=(0:x)$ and $D_{P'}x=Dx$. If $q\in Q\setminus D_{P'}$ then $q=sp^n/t$ where $s,t\in D\setminus P'$ and $n$ an integer $>0$. So, $pq\in D_{P'}$ if and only if $n=1$. It follows that $Dx=\{y\in E\mid py=0\}$. Let $\hat{a}$ and $\hat{b}$ be the respective multiplications in $R$ by $a$ and $b$. Then $\ker(\hat{a})=Rb$ and $\ker(\hat{b})=P$ which is not finitely generated. So, $\lambda_R(R/Ra)=2$ and $\lambda\mathrm{-dim}(R)=3$ by \cite[Theorem II.1]{Cou03}.
\end{proof}

\begin{theorem}
Let $R$ be a fqf-ring. Then $\lambda\mathrm{-dim}(R)\leq 3$.
\end{theorem}
\begin{proof}  Let $G$ be an $R$-module of finite $3$-presentation. So, there exists an exact sequence
\[F_3\xrightarrow{u_3} F_2\xrightarrow{u_2} F_1\xrightarrow{u_1} F_0\rightarrow G\rightarrow 0,\]
where $F_i$ is free of finite rank for $i=0,1,2,3$. Let $G_i$ be the image of $u_i$ for $i=1,2,3$. 

We do as in the proof of \cite[Theorem II.1]{Cou03}. For each maximal ideal $P$ we shall prove that there exist $t_P\in R\setminus P$ such that $\lambda_{R_{t_P}}(G_{t_P})\geq 4$. We end as in the proof of \cite[Theorem II.1]{Cou03} to show that $\ker(u_3)$ is finitely generated, by using the fact that $\mathrm{Max}\ R$ is a quasi-compact topological space.

 Let $P$ be a maximal ideal. First assume that $R_P$ is a chain ring. As in the proof of \cite[Theorem II.1]{Cou03} we show there exists $t_P\in R\setminus P$ such that $\lambda_{R_{t_P}}(G_{t_P})\geq 4$.

Now assume that $R_P$ is not a chain ring. We suppose that either $P$ is not minimal or $P$ is minimal but $PR_P$ is not finitely generated over $R_P$. In this case $(G_1)_P$ is free over $R_P$ by Proposition \ref{P:localdim}. Since $G_1$ is finitely presented, there exists $t_P\in R\setminus P$ such that $(G_1)_{t_P}$ is free over $R_{t_P}$ by \cite[Chapitre 2, \S 5, 1, Corollaire de la proposition 2]{Bou61}. It follows that $(G_2)_{t_P}$ and $(G_3)_{t_P}$ are projective. So, $\ker((u_3)_{t_P})$ is finitely generated. 

Finally assume that $R_P$ is not a chain ring, $P$ is minimal and $PR_P$ is finitely generated over $R_P$. We have $P^2R_P=0$. Since $P^2R_L=R_L$ for each maximal ideal $L\ne P$, $P^2$ is a pure ideal of $R$. It follows that $R/P^2$ is flat. Clearly $R/P^2$ is local . So, $R_P=R/P^2$. If $P^2$ is finitely generated then $P^2=Re$ where $e$ is an idempotent of $R$. So, if $t_P=1-e$  then $D(t_P)=\{P\}$, $R_{t_P}=R_P$ and $\ker((u_3)_{t_P})$ is finitely generated. If $P^2$ is not finitely generated then $P=I+P^2$ where $I$ is finitely generated but not principal because so is $P/P^2$. Since $I^2$ is a finitely generated subideal of the pure ideal $P^2$ there exists $a\in P^2$ such that $r=ar$ for each $r\in I^2$. It follows that $(1-a)I^2=0$. Hence $I^2R_t=0$ where $t=(1-a)$ and $IR_t\ne 0$ because for each $s\in I\setminus P^2$, $s\ne sa$. Since $G_t$ is finitely generated, after possibly multiplying $t$ with an element in $R\setminus P$, we may assume that $G_t$ has a generating system $\{g_1,\dots,g_p\}$ whose image in $(G_t)_P$ is a minimal generating system of $(G_t)_P$ containing $p$ elements. Let $F'_0$ be a free $R_t$-module with basis $\{e_1,\dots,e_p\}$, $\pi:F'_0\rightarrow G_t$ be the homomorphism defined by $\pi(e_k)=g_k$ for $k=1,\dots,p$ and $G'_1=\ker(\pi)$. We get the following commutative diagram with exact rows and columns:

\[\begin{matrix}
 {} & {} & 0 & {} & 0 & {} & 0 & {} & {}\\
{} & {} & \downarrow & {} & \downarrow & {} & \downarrow & {} & {} \\
0 & \rightarrow & P^2G'_1 & \rightarrow & P^2F'_0 & \xrightarrow {\pi}& P^2G_t & \rightarrow & 0 \\
 {} & {} & \downarrow & {} & \downarrow & {} & \downarrow & {} & {} \\
0 & \rightarrow & G'_1 & \rightarrow & F'_0 & \xrightarrow{\pi} & G_t & \rightarrow & 0 \\
{} & {} & \downarrow & {} & \downarrow & {} & \downarrow & {} & {} \\
0 & \rightarrow & (G'_1)_P & \rightarrow & (F'_0)_P & \xrightarrow{\pi_P} & (G_t)_P & \rightarrow & 0 \\
{} & {} & \downarrow & {} & \downarrow & {} & \downarrow & {} & {} \\
{} & {} & 0 & {} & 0 & {} & 0 & {} & {} 
\end{matrix}\]
Since $(G'_1)_P\subseteq P(F'_0)_P$, we have that $G'_1\subseteq PF'_0$. But, since $G'_1$ is finitely generated, after possibly multiplying $t$ with and element of $R\setminus P$, we may assume that $G'_1\subseteq IF'_0$ and that $G'_1$ has a generating system $\{g'_1,\dots,g'_q\}$ whose image in $(G'_1)_P$ is a minimal generating system of $(G'_1)_P$ with $q$ elements. Let $F'_1$ a free $R_t$-module with basis $\{e'_1,\dots,e'_q\}$, $u'_1:F'_1\rightarrow F'_0$ defined by $u'_1(e'_k)=g'_k$ for $k=1,\dots,q$ and $G'_2=\ker(u'_1)$. Again, for a suitable $t\in R\setminus P$ we may assume that $G'_2$ is contained in $IF'_1$ and has a generating system whose image in $(G'_2)_P$ is a minimal generating system with the same cardinal. Since $I_t^2=0$, it follows that $G'_2=IF'_1$. Consequently, if $I_t$ is generated by $\{r_1,\dots,r_n\}$ then $G'_2$ is generated $\{r_ie'_k\mid 1\leq i\leq n,\ 1\leq k\leq q\}$. Let $F'_2$ be a free $R_t$-module with basis $\{\epsilon_{i,k}\mid 1\leq i\leq n,\ 1\leq k\leq q\}$, $u'_2:F'_2\rightarrow F'_1$ be the homomorphism defined by $u'_2(\epsilon_{i,k})=r_ie'_k$ for $i=1,\dots,n$ and $k=1,\dots,q$ and $G'_3=\ker(u'_2)$. Since $G'_3$ is finitely generated, as above, for a suitable $t\in R\setminus P$, we get that $G'_3=I_tF'_2$. We easily deduce that $I_t=(0:_{R_t}r_i)$ for each $i=1,\dots,n$. Now, let $F'_3$ be a free $R_t$-module of rank $qn^2$ and $u'_3:F'_3\rightarrow F'_2$ be the homomorphism defined like $u'_2$. Then we get that $\ker(u'_3)=IF'_3$ and it is finitely generated. So, for a suitable $t_P\in R\setminus P$ we have $\lambda_{R_{t_P}}(G_{t_P})\geq 4$.  
\end{proof}

With a similar proof as the one of \cite[Corollary II.13]{Cou03}, and by using Proposition \ref{P:localdim} we get the following theorem.

\begin{theorem}\label{T:th2}
Let $R$ be a fqf-ring. Assume that $R_P$ is either an integral domain or a non-coherent ring for each maximal ideal $P$ which is not an isolated point of $\mathrm{Max}\ R$. Then $\lambda\mathrm{-dim}(R)\leq 2$.
\end{theorem}
\begin{proof}
Let $G$ be an $R$-module of finite $2$-presentation. So, there exists an exact sequence
\[F_2\xrightarrow{u_2} F_1\xrightarrow{u_1} F_0\rightarrow G\rightarrow 0,\]
where $F_i$ is free of finite rank for $i=0,1,2$. Let $G_i$ be the image of $u_i$ for $i=1,2$.

We do as in the proof of the previous theorem. First suppose that $P$ is a non-isolated point of $\mathrm{Max}\ R$. In this case $(G_1)_P$ is free over $R_P$ by Proposition \ref{P:localdim}. Since $G_1$ is finitely presented, there exists $t_P\in R\setminus P$ such that $(G_1)_{t_P}$ is free over $R_{t_P}$ by \cite[Chapitre 2, \S 5, 1, Corollaire de la proposition 2]{Bou61}. It follows that $(G_2)_{t_P}$ is projective. So, $\ker((u_2)_{t_P})$ is finitely generated. Now assume that $P$ is isolated. There exists $t_P\in R\setminus P$ such that $R_P\cong R_{t_P}$. By Theorem \ref{T:localdim} $\ker((u_2)_{t_P})$ is finitely generated.
\end{proof}

Example \ref{E:ex3} and the following show that the assumption "$R_P$ is a non-coherent ring" cannot be removed in Theorem \ref{T:th2}.

\begin{example}
Let $A$ be a von Neumann regular ring which is not self-injective, $H$ the injective hull of $A$, $x\in H\setminus A$, $E=A+Ax$ and $R=A\propto E$. Then:
\begin{enumerate}
\item  $R$ is a fqf-ring which is not an fqp-ring;
\item for each maximal ideal $P$, $R_P$ is Artinian;
\item $\lambda\mathrm{-dim}(R)=3$.
\end{enumerate} 
\end{example}
\begin{proof}
Let $N=\{(0, y)\mid y\in E\}$, $a=(0,1)$ and $b=(0,x)$.

$(1)$. See \cite[Example 4.6]{Cou15}.

$(2)$. If $P$ is a maximal ideal of $R$ then $R_P$ is the trivial ring extension of the field $A_{P'}$ by the finite dimensional vector space $E_{P'}$ where $P'=P/N$. Hence $R_P$ is Artinian.

$(3)$. Consider the following free resolution of $R/aR$:
\[R^2\xrightarrow{u_2} R\xrightarrow{u_1} R\rightarrow R/Ra\rightarrow 0\]
where $u_2((r,s))=ra+sb$ for each $(r,s)\in R^2$ and $u_1(r)=ra$ for each $r\in R$. We easily check that this sequence of $R$-modules is exact. The $A$-module $E$ is not finitely presented, else $E/A$ is finitely presented and, since each exact sequence of $A$-modules is pure, $A$ is a direct summand of $E$ which contradicts that $A$ is essential in $E$. Consequently, $N$, which is the image of $u_2$, is not finitely presented. So, $\lambda_R(R/Ra)=2$ and $\lambda\mathrm{-dim}(R)=3$. 
\end{proof}



\end{document}